\documentclass{elsarticle}

\usepackage{lineno,hyperref}
\usepackage{amsmath}
\usepackage{amssymb}
\usepackage{amsthm}
\usepackage{enumerate}

\modulolinenumbers[5]

\newtheorem{thm}{Theorem}[section]
\newtheorem{cor}[thm]{Corollary}
\newtheorem{lem}[thm]{Lemma}

\theoremstyle{definition}

\theoremstyle{remark}

\numberwithin{equation}{section}
\newcommand{\G}{{\mathcal G}}
\newcommand{\N}{{\mathcal N}}
\newcommand{\A}{{\mathcal A}}
\newcommand{\cP}{{\mathcal P}}

\newcommand{\T}{{\mathcal T}}

\newcommand{\Z}{{\mathcal Z}}

\usepackage{color}
\newcommand{\blue}{\textcolor{black}}

\begin{document}

\begin{frontmatter}

\title%[Characterising Tree-Based Networks]
{New Characterisations of Tree-Based Networks and Proximity Measures}

\author{Andrew Francis %\fnref{myfootnote}
}
\address{Centre for Research in Mathematics, School of Computing, Engineering and Mathematics, Western Sydney University, Sydney, Australia}
\ead{a.francis@westernsydney.edu.au}

\author{Charles Semple}
\address{School of Mathematics and Statistics, University of Canterbury, Christchurch, New Zealand}
\ead{charles.semple@canterbury.ac.nz}

\author{Mike Steel\corref{mycorrespondingauthor}}
\cortext[mycorrespondingauthor]{Corresponding author}
\address{School of Mathematics and Statistics, University of Canterbury, Christchurch, New Zealand}
\ead{mike.steel@canterbury.ac.nz}

\begin{abstract}
Phylogenetic networks are a type of directed acyclic graph that represent how a set $X$ of present-day species are descended from a common ancestor by  processes of speciation and reticulate evolution. In the absence of reticulate evolution, such networks are simply phylogenetic (evolutionary) trees. Moreover,    phylogenetic networks that are not trees can sometimes be represented as phylogenetic trees with additional directed edges placed between their edges. Such networks are called {\em tree-based}, and the class of phylogenetic networks that are tree-based has recently been characterised. In this paper,  we establish a number of new characterisations of tree-based networks in terms of path partitions and antichains (in the spirit of  Dilworth's theorem), as well as via  matchings in a bipartite graph.   We also show that a temporal network is tree-based if and only if it satisfies an antichain-to-leaf condition. In the second part of the paper, we define three indices that measure the extent to which an arbitrary phylogenetic network deviates from being tree-based. We describe how these three indices can be described exactly and computed efficiently using classical results concerning maximum-sized matchings in bipartite graphs.
\end{abstract}


\end{frontmatter}

%\linenumbers

\section{Introduction}
 
Phylogenetic networks are of increasing interest in the literature as they allow for the representation of reticulate (non-tree-like) processes in evolution. From a mathematical perspective, there are several particularly attractive classes of phylogenetic networks. One of those classes is tree-based networks. Intuitively, a phylogenetic network is tree-based if it can be obtained from a phylogenetic tree $\T$ by simply adding edges whose end-vertices subdivide edges of $\T$. Formalised and studied in~\cite{fra15}, tree-based networks have since been studied in a number of recent papers~\cite{ana16, hay16, jet16, sem16, zha16}, in a variety of contexts.

In this paper, we establish several new characterisations of tree-based networks. These characterisations are based on antichains, path partitions, and matchings in bipartite graphs, and complement the previous characterisations based on bipartite matchings~\cite{jet16, zha16}. Furthermore, with the aid of these characterisations, we explore indices quantifying the closeness of an arbitrary phylogenetic network $\N$ to being tree-based. Each of the considered indices is computable in time polynomial in the size of $\N$, \blue{that is, in the number of vertices in $\N$,} by finding a maximum-size matching in certain bipartite graphs based on $\N$.

The paper is organised as follows. The rest of the introduction contains some formal definitions and previous results. In Section~\ref{characterisations}, we state the new characterisations of tree-based networks as well as a characterisation of tree-based networks within the class of temporal networks. Except for one characterisation which is proved in Section~\ref{deviation}, the proofs of these characterisations are given in Section~\ref{proofs}. In Section~\ref{deviation}, we consider three indices quantifying the extent to which an arbitrary phylogenetic network is tree-based. We end the paper with a brief conclusion and further questions in Section~\ref{ending}.

%The path decompositions of digraphs is considered further in  \cite{als74}.

\subsection{Definitions}

Throughout the paper, $X$ denotes a non-empty finite set. A {\em phylogenetic network $\N$ on $X$} is a rooted acyclic digraph with no edges in parallel and satisfying the following properties:
\begin{enumerate}[(i)]
\item the (unique) root has out-degree two;

\item a vertex with out-degree zero has in-degree one, and the set of vertices with out-degree zero is $X$; and

\item all other vertices either have in-degree one and out-degree two, or in-degree two and out-degree one.
\end{enumerate}
For technical reasons, if $|X|=1$, we additionally allow $\N$ to consist of the single vertex in $X$. The vertices of $\N$ with out-degree zero are referred to as {\em leaves}. Furthermore, the vertices of $\N$ with in-degree two and out-degree one are {\em reticulations}, while the root together with the vertices with in-degree one and out-degree two are {\em tree vertices}. The edges directed into a reticulation are {\em reticulation edges}; all other edges are {\em tree edges}. In the literature, a phylogenetic network as defined here is sometimes called a {\em binary} phylogenetic network. \blue{A {\em rooted (binary) phylogenetic $X$-tree} is a phylogenetic network on $X$ with no reticulations.}

A phylogenetic network $\N = (V,E)$ on $X$ is a {\em tree-based network} if $\N$ has a rooted spanning tree $(V, E')$, where $E'\subseteq E$, that has all its leaves in $X$. We refer to this spanning tree as a {\em base tree} for $\N$. The initial study of tree-based networks~\cite{fra15} included the presentation of a polynomial-time algorithm (based on a characterisation involving 2-SAT) for determining whether or not an arbitrary phylogenetic network is tree-based and, if so, constructing a base tree for it. Fig.~\ref{f:killer} shows an example of a phylogenetic network that is tree-based, and one that is not. As with all drawings of phylogenetic networks in this paper, edges are directed down the page \blue{and, thus, away from from the root}.

\begin{figure}
\begin{center}
\includegraphics[width=11cm]{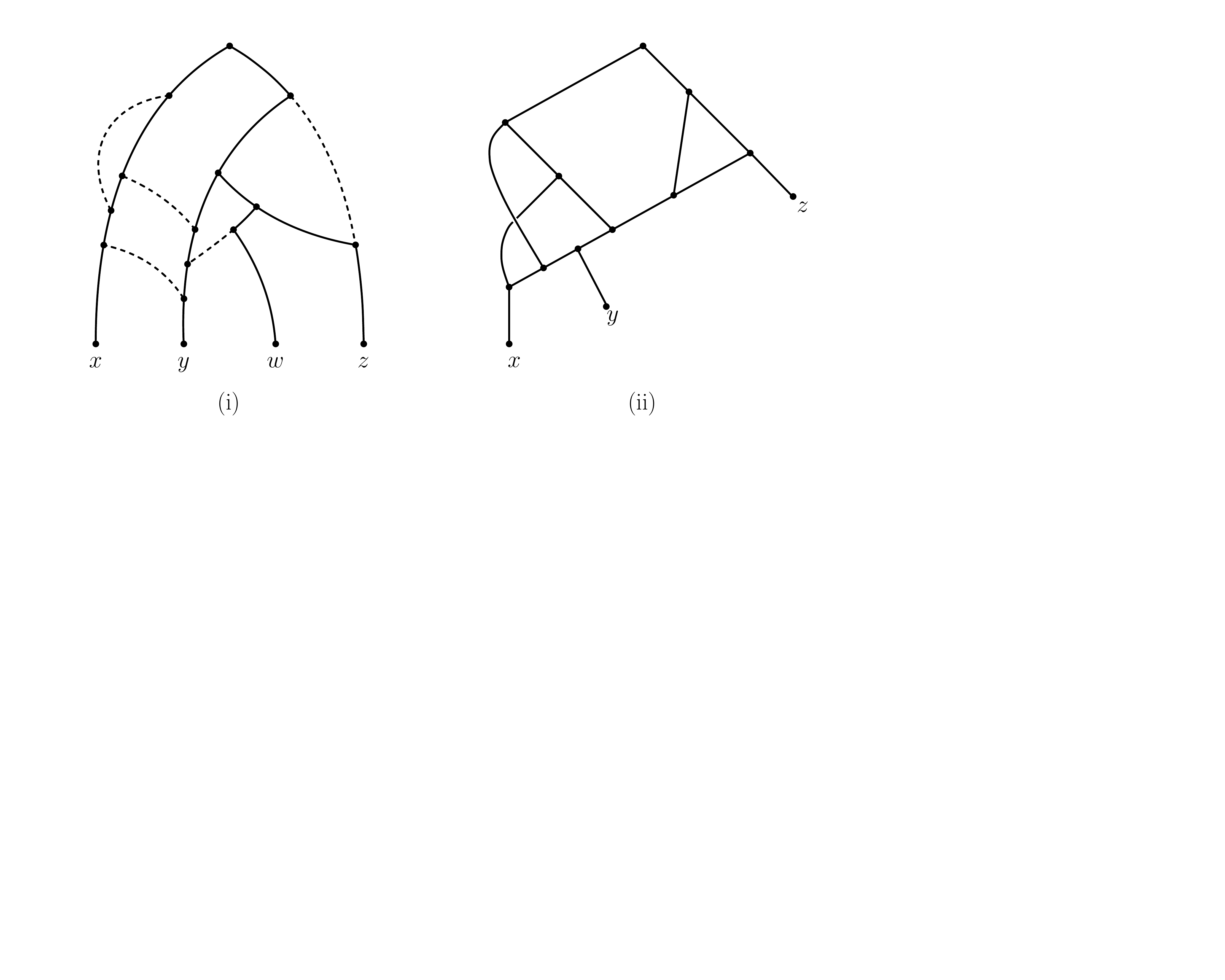}
\caption{Examples of (i) a tree-based network, showing a base tree (\blue{solid} edges), and (ii) a non-tree-based network.}
\label{f:killer}
\end{center}
\end{figure}

\blue{Lastly, in various places in the paper, we use the following operation. Let $G$ be a directed graph and let $e$ be an edge of $G$. The operation of deleting $e$, adding a new vertex $x$, and adjoining $x$ to each of the end-vertices of $e$ (orienting the new edges in the original direction of $e$) is to {\em subdivide $e$}. Any directed graph obtained from $G$ by a sequence of edge subdivisions is called a {\em subdivision} of $G$.}

\subsection{Previous results}
\label{sec:prev}

Let  $\N$ be a phylogenetic network on $X$. Let $T$ be the set of tree vertices in $\N$ that are parents of a reticulation and let $R$ be the set of reticulations in $\N$. Zhang~\cite{zha16} defined the following bipartite graph to characterise tree-based networks. Let $\Z_{\N}$ be the graph with vertex set $T\cup R$ and edge set
$$\{\{t, r\}: \mbox{$t\in T$, $r\in R$, and $(t, r)$ is an edge in $\N$}\}.$$
In particular, Zhang~\cite{zha16} established the following characterisations.

\begin{thm}
Let $\N$ be a phylogenetic network. Then the following are equivalent:
\begin{enumerate}[{\rm (i)}]
\item $\N$ is tree-based.

\item The bipartite graph $\Z_{\N}$ has a matching such that each reticulation is matched.

\item The bipartite graph $\Z_{\N}$ has no maximal path that starts and ends with reticulations.
\end{enumerate} 
\label{zhang}
\end{thm}

\noindent The proofs of our main results make use of these characterisations. It is clear that (ii) implies (iii), while the reverse implication relies on Hall's theorem\blue{~\cite{hal35} (see \cite[Theorem~16.4]{bon08})} on matchings in bipartite graphs.

There are several other characterisations of what it means for a phylogenetic network $\N=(V, E)$ to be a tree-based network:
\begin{enumerate}[(i)]
\item The existence of an `admissible' subset of the edges of $\N$ \cite{fra15}.

\item The existence of an independent subset of edges $E'$ of $E$ for which $(V,E-E')$ is a base tree of $\N$ \cite{fra15}.

\item The existence of a matching in a certain bipartite graph whose vertices consist of  `omnians' (vertices whose only children are reticulations) and reticulations \cite{jet16}.
\end{enumerate}
While we mention these characterisations, they will play no further part in the paper.

\section{New Characterisations}
\label{characterisations}

\blue{An {\em antichain} in a directed graph is a subset $S$ of vertices with the property that, for all distinct $u, v\in S$, there is no directed path from $u$ to $v$.} Any tree-based network $\N$ satisfies the {\em antichain-to-leaf property}, which says that for any antichain of $k$ vertices, there \blue{exist} $k$ vertex disjoint paths from the elements of the antichain to the leaves of $\N$~\cite{fra15}. However, this property is not sufficient to ensure a network is tree-based, as the counterexample in Fig.~\ref{f:killer}(ii) illustrates.  
Note that if an arbitrary phylogenetic network satisfies the antichain-to-leaf property, then it satisfies the corresponding property for \emph{edge} disjoint paths, and in fact is equivalent to it via an application of Menger's theorem\blue{~\cite{men27} (see \cite[Theorem~7.16]{bon08})}.

The main result of this section is the following theorem. This theorem provides five properties of a phylogenetic network $\N=(V, E)$ that are equivalent to being tree-based. Four of these properties, (II)--(V), can be viewed as providing different ways to strengthen the antichain-to-leaf property. The final property, (VI), provides another characterisation in terms of bipartite graphs. Let $\G_\N$ denote the bipartite graph whose vertex bipartition is $\{V_1, V_2\}$, where each of $V_1$ and $V_2$ is a copy of $V$, and with an edge joining a vertex $u\in V_1$ and a vertex $v\in V_2$ precisely if $(u, v)$ is an edge in $\N$. To illustrate, consider the phylogenetic network $\N$ shown in Fig.~\ref{f:match}(i). The bipartite graph $\G_{\N}$ is shown in Fig.~\ref{f:match}(ii).

\begin{thm}
\label{t:tbn.char}
Let $\N=(V, E)$ be a phylogenetic network on $X$. The following are equivalent:
\begin{enumerate}[{\rm (I)}]
\item $\N$ is tree-based;

\item $\N$ has an antichain $\A\subseteq V$, and a partition $\Pi$ of $V$ into $|\A|$ chains each of which forms a path in $\N$ ending at a leaf in $X$;

\item For all $U\subseteq V$, there exists a set of vertex disjoint paths in $\N$ each ending at a leaf in $X$ such that each element of $U$ is on exactly one path;

\item There is no pair of subsets $U_1, U_2\subseteq V$ such that $|U_1|>|U_2|$ and
\begin{enumerate}[{\rm (i)}]
\item every path from a vertex in $U_1$ to a vertex in $X$ traverses a vertex in $U_2$, and

\item for $\{i, j\}=\{1, 2\}$, if there is a path from a vertex in $U_i$ to a vertex in $U_i$, then this path traverses a vertex in $U_j$.
\end{enumerate}

\item The vertex set of $\N$ can be partitioned into a set of vertex disjoint paths, each of which ends at a leaf in $X$.

\item The bipartite graph $\G_\N$ has a matching of size $|V|-|X|$.
\end{enumerate}
\end{thm}

{\bf Remarks:}
\begin{itemize}
\item Property (II) is related to a classical result in combinatorics, namely Dilworth's theorem\blue{~\cite{dil50}}. A version of this theorem states that every finite poset $P$ has an antichain $A$ and a partition $\Pi$ of $P$ into $|A|$ chains \blue{(see \cite[Theorem~19.5]{bon08})}.  If we regard a network as a poset in the usual way, the additional requirement in (II) is that the chains must form paths; one cannot simply `jump over' other vertices.

\noindent It is instructive to see why this fails for the counterexample to the `antichain-to-leaf property' in \cite{fra15}, reproduced in Fig.~\ref{f:killer}(ii).  This network has a maximum-sized antichain of size three (e.g. $\{x,y, z\}$), and so the vertex set can be partitioned into three chains (by Dilworth's theorem).  However, the network cannot be partitioned into three paths. A partition into four paths is shown in Fig.~\ref{f:two_killer}.

\item Notice that the condition in (II) forces $|\A|=|X|$, since each element $x$ in $X$ has to be in exactly one path from $\Pi$. But, $\A$ need not necessarily be equal to $X$, indeed one could have $\A\cap X=\emptyset$.

\item Property (VI) complements the two existing characterisations of tree-based networks via matchings \cite{jet16, zha16}.
\end{itemize}

\begin{figure}
\begin{center}
\includegraphics[width=5.5cm]{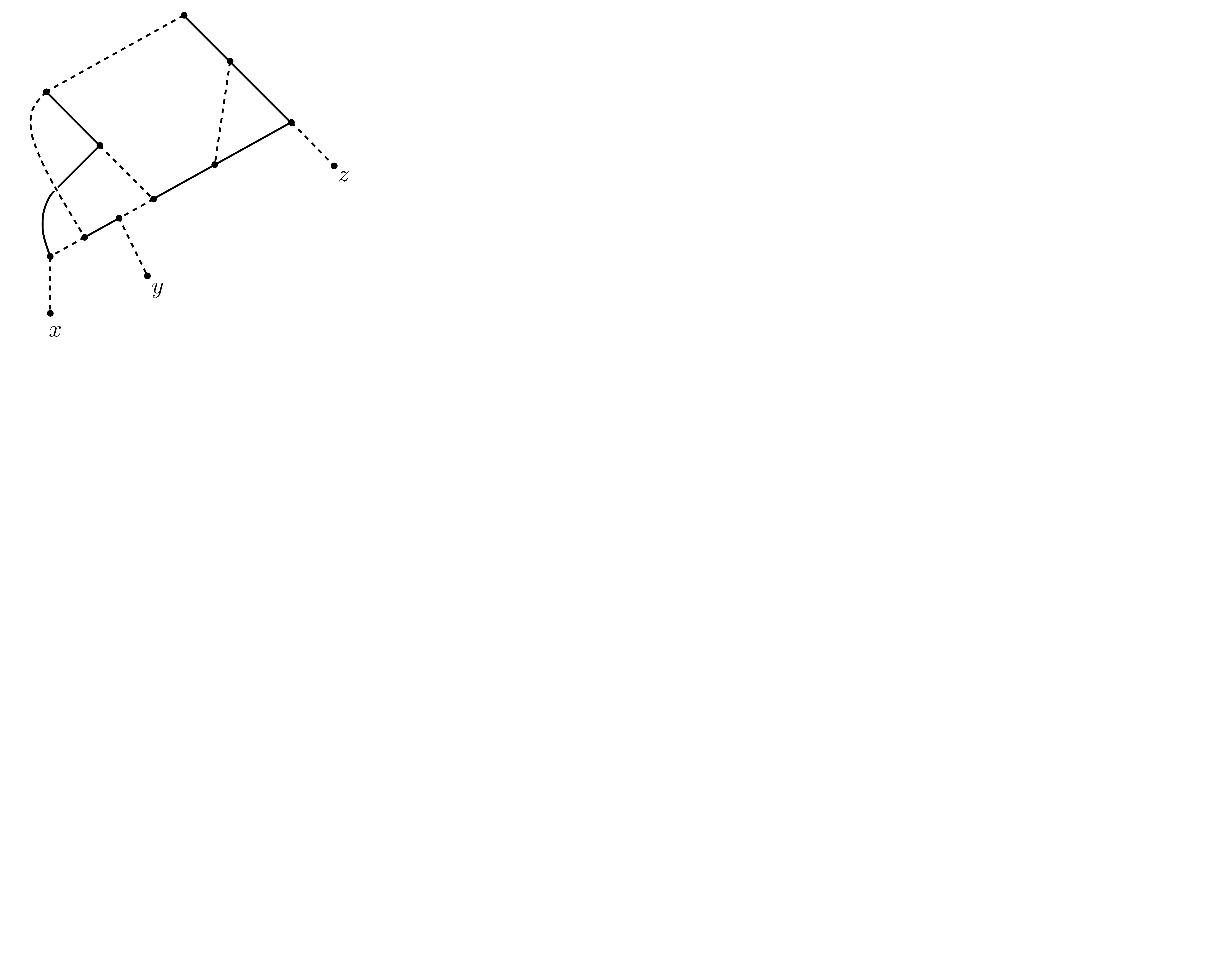}
\caption{For the non-tree-based \blue{network} from Fig.~\ref{f:killer}(ii), the vertex set can be partitioned into four paths (\blue{dashed}), but not into three (the size of the maximum-sized antichain).}
\label{f:two_killer}
\end{center}
\end{figure}

A second result in this section is a characterisation of tree-based networks within the class of `temporal' networks. Introduced in~\cite{bar06}, a phylogenetic network $\N=(V,E)$ is {\em temporal} if there is a map $\lambda: V \rightarrow {\mathbb R}$ so that $\lambda(u) < \lambda(v)$ for each tree edge $(u, v)$, and $\lambda(u) = \lambda(v)$ for each reticulation edge $(u, v)$, in which case $\lambda$ is a {\em temporal map for $\N$}. As illustrated above, the antichain-to-leaf property is a necessary but not sufficient condition for a phylogenetic network to be tree-based. However, within the class of temporal networks, it is sufficient.

\begin{thm}
\label{t:temporal}
Let $\N$ be a temporal network. Then $\N$ is tree-based if and only if $\N$ satisfies the antichain-to-leaf property.
\end{thm}

\section{Proof of Theorems~\ref{t:tbn.char} and \ref{t:temporal}}
\label{proofs}

In this section, we establish the equivalence of (I)--(V) in Theorem~\ref{t:tbn.char} as well as Theorem~\ref{t:temporal}. The proof of the equivalence of (I) and (VI) in Theorem~\ref{t:tbn.char} is done independently as Corollary~\ref{matchcor}. We begin with a lemma.

\begin{lem}
\label{l:trees.M}
Let $T$ be a subdivision of a rooted binary tree with vertex set $V$. Then the following property holds:
\begin{enumerate}[{\rm (P)}]
\item For any non-empty subset $U$ of $V$ there exists a set of vertex disjoint (directed) paths in $T$ each of which ends at a leaf of $T$ and each vertex in $U$ lies on exactly one path.
\end{enumerate}
\end{lem}

\begin{proof}
We apply induction on the number $n$ of vertices of $T$. For $n=1$, (P) trivially holds. Suppose that $n\ge 2$ and (P) holds for all subdivisions of a rooted binary tree with at most $n-1$ vertices. Let $U$ be an arbitrary subset of vertices of $T$. Since $n\ge 2$, it follows that $T$ either has
\begin{enumerate}[(i)]
\item a leaf $x$ whose parent, $u$ say, has degree~$2$, or

\item a vertex $v$ that is a parent of two leaves, $x$ and $y$ say.
\end{enumerate}
We establish the induction hypothesis in each case, starting with (i).

For (i), let $T'$ be the subdivision of a rooted binary tree obtained from $T$ by deleting $x$ and its incident edge, so that $u$ is now a leaf of $T'$. Let
\begin{align*}
U' =
\begin{cases}
U, & \text{if $U$ does not contain $x$;} \\
U-\{x\}, & \text{if $U$ contains $x$ and also contains $u$;} \\
(U-\{x\})\cup \{u\}, & \text{if $U$ contains $x$ but not $u$.}
\end{cases}
\end{align*}
Observe that $U'$ is a subset of vertices of $T'$. Therefore, as $T'$ has $n-1$ vertices, it follows by induction that (P) holds and so there is a set of disjoint paths in $T'$ each of which ends at a leaf of $T'$ and each vertex in $U$ lies on exactly one path. Now one of these paths ends at $u$. Replacing this path with the one that extends it to end at $x$ gives a set of vertex disjoint paths in $T$, each of which ends at a leaf of $T$, and each vertex in $U$ lies on exactly one path. Thus the lemma holds for (i).

Now consider (ii). Let $T'$ be the subdivision of a rooted binary tree obtained from $T$ by deleting $y$ and its incident edge. Note that $T'$ has $n-1$ vertices. If $U$ does not contain $y$, then let $U'=U$. By induction, there is a set of vertex disjoint paths in $T'$ each of which ends at a leaf of $T'$ and each vertex in $U'$ lies on exactly one path. This set of paths also \blue{works} for $U$ in $T$. On the other hand, if $U$ does contain $y$, then let $U'=U-\{y\}$. By induction, there is a set of at most $|U'|=|U|-1$ vertex disjoint paths in $T'$ each of which ends at a leaf of $T'$ and each vertex in $U'$ lies on exactly one path. Adding the (trivial) path consisting of just $y$ to this set of paths, we obtain a set of vertex disjoint paths in $T$ each ending at a leaf of $T$ and each vertex in $U$ lying on exactly one path. This completes the proof for (ii).
\end{proof}

\begin{proof}[Proof of equivalence of {\rm (I)--(V)} in Theorem~\ref{t:tbn.char}]
We establish the following implications between the stated conditions on $\N$: (I) implies (II), (II) implies (III), (III) implies (IV), and (IV) implies (I), which together show the equivalence of (I)-(IV).  We then show that (III) implies (V) and (V) implies (II).

\noindent (I) $\Rightarrow$ (II). Suppose that $\N$ is tree-based, and $T$ is a base tree of $\N$. Let $U= V-X$. Then, by Lemma~\ref{l:trees.M}, there is a collection of vertex disjoint paths in $T$ each ending at a vertex in $X$ and each vertex in $\N$ lying on exactly one path. Choosing $\A=X$, the vertex sets of these paths form the blocks of the required partition $\Pi$ of $V$.

\noindent (II) $\Rightarrow$ (III). Suppose that $\Pi$ is a partition of $V$ with the property that each block in $\Pi$ is the vertex set of a path in $\N$ ending at a leaf in $X$. Let $U$ be a subset of vertices of $\N$. Then $\Pi$ provides a set of vertex disjoint paths each ending at a leaf in $X$ and with each vertex in $U$ on exactly one path.

\noindent (III) $\Rightarrow$ (IV). For this implication, we prove the contrapositive. Suppose that property (IV) is false for $\N$. Then there exist subsets $U_1$ and $U_2$ of $V$ with $|U_1| > |U_2|$ that satisfy the two traversal conditions (i) and (ii). We show $U=U_1$ fails to satisfy property (III). First observe that if $P$ is a path in $\N$ ending at $X$, then $P$ contains at least as many vertices of $U_2$ as $U_1$. To see this, observe that because of the traversal conditions (i) and (ii), as we move along $P$, we alternate between vertices in $U_1$ and vertices $U_2$. That is, for $\{i, j\}=\{1, 2\}$, if we traverse a vertex in $U_i$, then the next vertex we traverse in $U_i\cup U_j$ is a vertex in $U_j$. Moreover, for each vertex in $U_1$ on $P$, there is a subsequent vertex in $U_2$ on $P$.  Hence there are at least as many vertices of $U_2$ as $U_1$ in $P$. Thus any set of vertex disjoint paths in $\N$ each ending at a leaf in $X$ collectively contains at least as many vertices in $U_2$ as $U_1$. But then it is not possible for such a set of paths to collectively contain all the vertices in $U_1$ since $|U_2| < |U_1|$. So $U=U_1$ violates property (III).

\noindent (IV) $\Rightarrow$ (I). Again, we prove the contrapositive. Suppose that $\N$ is not tree-based. Then, by Theorem~\ref{zhang}, there is a maximal path in $\Z_{\N}$ that starts and ends in $R$. Writing this path as $r_1\, t_1\, r_2\, \cdots\, t_{k-1}\, r_k$, let $q$ be the parent of $r_1$ that is not $t_1$ and let $q'$ be the parent of $r_k$ that is not $t_{k-1}$. Since the path is maximal, both $q$ and $q'$ are reticulations of $\N$. Let $U_1 = \{q, t_1, t_2, \ldots, t_{k-1}, q'\}$ and $U_2 = \{r_1, r_2, \ldots, r_k\}$. These sets have the following properties:
\begin{enumerate}[(1)]
\item $|U_1|>|U_2|$,

\item $U_1$ is the set of all parents of all vertices in $U_2$, and

\item $U_2$ is the set of all children of all vertices in $U_1$.
\end{enumerate}
Now (2) implies that every path from a vertex in $U_1$ to a vertex in $X$ traverses a vertex in $U_2$. Furthermore, (2) and (3) imply that, if there is a path from a vertex in $U_1$ to another vertex in $U_1$, then this path traverses a vertex in $U_2$. Similarly, if there is a path from a vertex in $U_2$ to another vertex in $U_2$, then this path traverses a vertex in $U_1$. It follows that $U_1$ and $U_2$ provide an instance for which property (IV) fails.

\noindent (III) $\Rightarrow$ (V). Taking $U=V$ in property (III) immediately gives a path system satisfying property (V).

\noindent (V) $\Rightarrow$ (II). Suppose $P$ is a set of paths satisfying property (V). Taking $\A = X$ and $\Pi = P$ gives an antichain and partition of $V$ into $|\A|$ chains that satisfies property (II).
\end{proof}

\begin{proof}[Proof of Theorem~\ref{t:temporal}]
The `only if' direction holds even without the temporal condition, and was established in \cite{fra15}. It also follows from the equivalence of (I) and (III) in Theorem~\ref{t:tbn.char}. For the `if' direction, suppose that $\N$ is temporal but not tree-based. We will show that $\N$ does not possess the antichain-to-leaf property.

Since $\N$ is not tree-based, it follows by Theorem~\ref{zhang} that the bipartite graph $\Z_{\N}$ contains a maximal path
$$r_1\, t_1\, r_2\, \cdots\, t_{k-1}\, r_k$$
that starts and ends in $R$. If $k=1$, then both parents, $q$ and $q'$ say, of $r_1$ are reticulations, in which case $U=\{q, q'\}$ is an antichain in $\N$ that violates the antichain-to-leaf property. Thus we may assume that $k\ge 2$.

Let $q$ be the parent of $r_1$ that is not $t_1$ and let $q'$ be the parent of $r_k$ that is not $t_{k-1}$. Since the path is maximal, both $q$ and $q'$ are reticulations in $\N$. Let $U= \{q, t_1, t_2, \ldots, t_{k-1}, q'\}$. Since $\N$ is temporal, there is a temporal map $\lambda$ for $\N$ which necessarily gives
$$\lambda(q)=\lambda(r_1)=\lambda(t_1)=\lambda(r_2)= \cdots = \lambda(t_{k-1}) = \lambda(r_k)=  \lambda(q'),$$
and so $\lambda$ is constant on $U$. If $U$ is an antichain, then $U$  violates the antichain-to-leaf property, since any set of paths that connects the $k+1$ vertices in $U$ to the leafs in $X$ need to pass through the $k$ vertices in $\{r_1, r_2 \ldots, r_k\}$ and so these paths cannot be disjoint.

Therefore, suppose that $U$ is not an antichain. Then there is a directed path $P$ in $\N$ from a vertex $u\in U$ to another vertex $u'\in U$. Moreover, every edge in $P$ must be a reticulation edge of $\N$. Otherwise, if $P$ contains a tree edge, then $\lambda(u') > \lambda(u)$, contradicting the constancy of $\lambda$ on $U$.  In particular, the only possible tree \emph{vertex} of $\N$ in $P$ is the first vertex. Thus $P$ must include $q$ or $q'$. If $q$ (resp.\ $q'$) can be reached by a directed path from a vertex in $\{r_1, r_2, \ldots, r_k\}$, denote this vertex by $r_q$ (resp.\ $r_{q'}$). It is easily checked that $r_q\neq r_{q'}$. Now let
$$U'=
\begin{cases}
U-\{q\}, & \mbox{if $r_q$ exists}; \\
U-\{q'\}, & \mbox{if $r_{q'}$ exists}; \\
U-\{q, q'\}, & \mbox{if $r_q$ and $r_{q'}$ exist}.
\end{cases}
$$
The set $U'$ is an antichain of size $k$ or $k-1$. Now any path in $\N$ that connects a vertex in $U'$ with a vertex in $X$ must traverse a vertex in $\{r_1, r_2, \ldots, r_k\}$. But if $r_q$ exists, then any path traversing $r_q$ must also traverse $r_1$. Similarly, if $r_{q'}$ exists, then any path traversing $r_{q'}$ must also traverse $r_{k-1}$. In all possibilities for $U'$, it follows that $U'$ does not satisfy the antichain-to-leaf property.
\end{proof}

\section{Measures of Deviation}
\label{deviation}

The concept of being tree-based is an `all-or-nothing' property. In this section, we consider computable indices associated with a phylogenetic network $\N$ which are each zero if and only if $\N$ is tree-based. These indices provide, more generally, some measure of how close an arbitrary phylogenetic network is to being tree-based.

Let $\N=(V, E)$ be a phylogenetic network on $X$. Consider the operation of adjoining a new leaf $y$ to $\N$ by subdividing an edge of $\N$ with a new vertex, $u$ say, and adding the edge $(u, y)$. Observe that $u$ is a tree vertex in the resulting network. We refer to this operation as {\em attaching a new leaf} to $\N$. The three measures we consider are as follows:
\begin{enumerate}[(i)]
\item The minimum number $l(\N)$ of leaves in \blue{$V\setminus X$} that must be present as leaves in a rooted spanning tree of $\N$.

\item The minimum number \blue{$p(\N)=d(\N)-|X|$, where $d(\N)$ is the smallest number} of vertex disjoint paths that partition the vertices of $\N$.

\item The minimum number $t(\N)$ of leaves that need to be attached to $\N$ so the resulting network is tree-based.

%\item The minimal number of vertices that are ``jumped'' by the chains provided by Dilworth's Theorem.  That is, the number of obstructions to the condition given in Theorem~\ref{t:tbn.dilworth}. tree based: 0.
\end{enumerate}
Each of these measures is non-negative and well defined. To see that (iii) is well defined, attach a new leaf to each reticulation edge in $\N$. It follows by Theorem~\ref{zhang} that the resulting network is tree-based. Moreover, each of these measures equal zero if and only if $\N$ is tree-based. For (ii), this relies on one direction of the equivalence of (I) and (II) in Theorem~\ref{t:tbn.char}.

We will show that each of the measures are computable in time polynomial in the size of $\N$. \blue{Unexpectedly,} it turns out that \blue{(i), (ii) and (iii)} are identical and can be computed by finding a maximum-sized matching in $\G_{\N}$.

% while (iii) can be computed by finding a maximum-sized matching in $\Z_{\N}$.

Recall that, for a phylogenetic network $N=(V, E)$, the bipartite graph $\G_{\N}$ has vertex bipartition $\{V_1, V_2\}$, where $V_1$ and $V_2$ are copies of $V$, and an edge joining a vertex $u\in V_1$ with a vertex $v\in V_2$ precisely if $(u, v)$ is an edge in $\N$. Relative to a maximum-sized matching of $\G_{\N}$, let $u(\G_{\N})$ denote the number of unmatched vertices of $V_1$. Note that, as each of the elements in $X$ are isolated vertices in $V_1$, each of these elements is unmatched regardless of the matching. The first result of this section equates $p(\N)$ to the number vertices unmatched by a maximum-sized matching of $\G_{\N}$. An illustration of the proof of this result is given in Fig.~\ref{f:match}.

\begin{figure}
\begin{center}
\includegraphics[width=8cm]{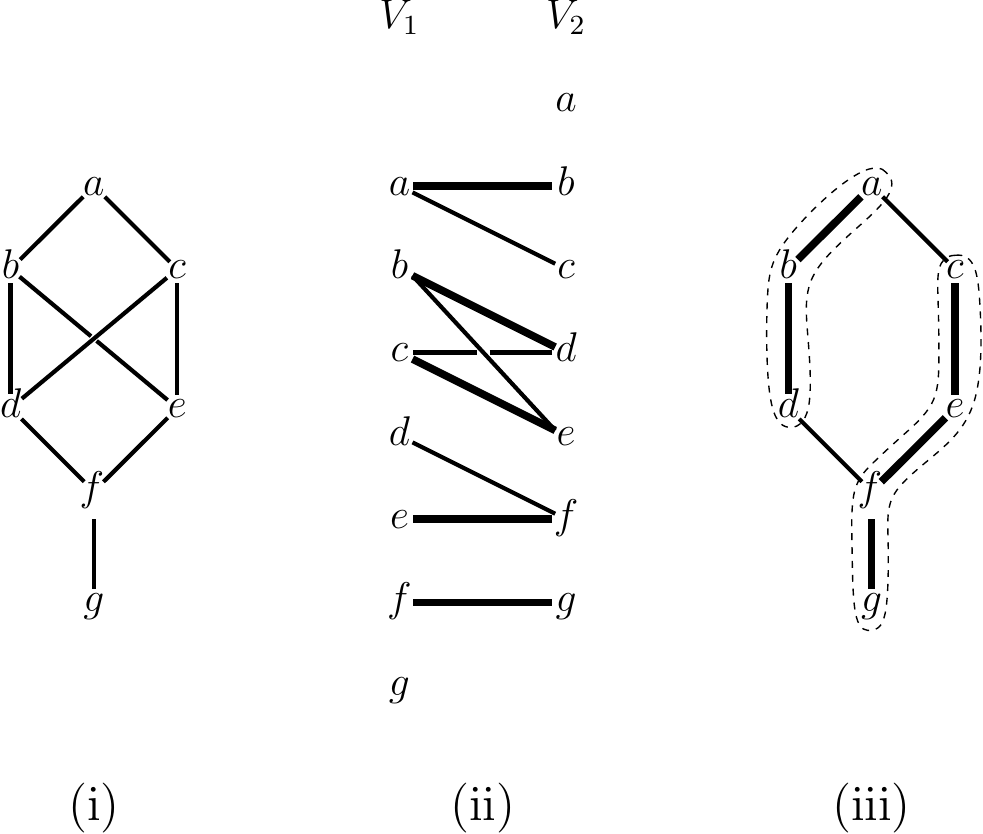}
\caption{(i) A phylogenetic network $\N$ on $X$ that is not tree-based, and (ii) the bipartite graph $\G_{\N}$. A maximum-sized matching of $\G_{\N}$ is indicated by the bold edges, and the two corresponding vertex disjoint paths in $\N$ are indicated in (iii). For this example, \blue{$u(\G_{\N})=2$}, and so $p(\N)=u(\G_{\N})-|X|=2-1=1$. A rooted spanning tree with one leaf not in $X$ is obtained from the two paths in (iii) by adding the edge $(a, c)$.}
\label{f:match}
\end{center}
\end{figure}

\begin{lem}
Let $\N$ be a phylogenetic network on $X$. Then
$$p(\N)=u(\G_{\N})-|X|.$$
\label{unmatched}
\end{lem}

\begin{proof}
We first show that $p(\N)\le u(\G)-|X|$. Let $M$ be a matching of $\G_{\N}$. Let $U_2$ denote the set of unmatched vertices in $V_2$. For each vertex $u\in U_2$, we recursively construct a directed path $P_u$ in $\N$ as follows. Set $u=u_0$ and initially set $P_u=u_0$. If $u_0$ is unmatched in $V_1$, then terminate the process and set $P_u=u_0$; otherwise, $u_0$ is matched in $V_1$, in which case set $P_u=u_0\, u_1$, where $(u_0, u_1)\in M$. If $u_1$ is unmatched in $V_1$, then terminate the process and set $P_u=u_0\, u_1$. Otherwise, $u_1$ is matched in $V_1$, in which case set $P_u=u_0\, u_1\, u_2$, where $(u_1, u_2)\in M$. Since $\N$ is acyclic, this process eventually terminates with the last vertex, $u_k$ say, added to $P_u$ being unmatched in $V_1$.

Repeating this construction for each vertex in $U_2$, we eventually obtained a collection $\cP=\{P_u: u\in U_2\}$ of directed paths in $\N$. Since $M$ is a matching, the paths in $\cP$ are vertex disjoint. Furthermore, every vertex in $\N$ is on some path in $\cP$. To see this, suppose there is a vertex $v\in V$ not on a path in $\cP$. Clearly, $v$ is matched in $V_2$. But then, by reversing the above construction starting at $v$ in $V_2$, it is easily seen that $v$ is on such a path. Since each vertex in $X$ is unmatched in $V_1$, and noting that the number of paths in $\cP$ equates to the number of unmatched vertices in $V_2$, and therefore the number of unmatched vertices in $V_1$, it follows by choosing $M$ to be of maximum size that
$$p(\N)\le |\cP|-|X|=u(\G_{\N})-|X|.$$

We next show that $p(\N)\ge u(\G_{\N})-|X|$. Now let $\cP$ be a collection of vertex disjoint paths that partitions the vertices of $\N$. Let $M$ be the matching of $\G_{\N}$ obtained from $\cP$ as follows. The edge $(u, v)\in M$ precisely if $u$ and $v$ are consecutive vertices on some path in $\cP$. Since the paths in $\cP$ are vertex disjoint, $M$ is certainly a matching. As every vertex in $\N$ is on some path in $\cP$, the number $u_1$ of unmatched vertices in $V_1$ is the number of paths in $\cP$, each such vertex is the last vertex of some path in $\cP$. Thus, by choosing $\cP$ to be of minimum size,
$$p(\N)=|\cP|-|X|=u_1-|X|\ge u(\G_{\N})-|X|.$$
This completes the proof of the lemma.
\end{proof}

The next corollary establishes the equivalence of (I) and (VI) in Theorem~\ref{t:tbn.char}.

\begin{cor}
\label{matchcor}
A phylogenetic network $\N$ on $X$ is tree-based if and only if $\G_\N$ has a matching of size $|V|-|X|$. 
\end{cor} 

\begin{proof}
Since the elements of $X$ are always unmatched in $V_1$, it follows that $\G_{\N}$ has a matching of size
$|V|-|X|$ if and only if $\G_{\N}$ has a maximum-sized matching of this size. In turn, by Lemma~\ref{unmatched}, the latter holds if and only if  $p(\N) = u(\G_{\N}) -|X|=0$. Noting that $p(\N)=0$ if and only if $\N$ is tree-based completes the proof.
\end{proof}

We now show that \blue{the three} measures are identical. An illustration of \blue{the first part of} the proof of the next theorem is given in Fig.~\ref{f:match}.

\begin{thm}
\label{matchthm}
Let $\N$ be a phylogenetic network $\N$ on $X$. Then
$$l(\N)=p(\N)=t(\N).$$
\end{thm}

\begin{proof}
%We first show that $p(\N) \leq l(\N)$. Let $U$ be the set of leaves of a rooted spanning tree $T$ that realises $l(\N)$. Since $U$ contains $X$, we have $|U|= l(\N) + |X|$. If we now apply Lemma~\ref{l:trees.M} with the same choice of $U$ and $T$, then $T$ can be partitioned into at most $l(\N)+|X|$ paths each of which ends at an element of $U$. Thus $p(\N)\leq l(\N)$.

\blue{We first} show that $l(\N)\leq p(\N)$. Suppose that $\Pi$ is the partition of the vertex set $V$ of $\N$ induced by a set of $p(\N)+|X|$ vertex disjoint paths of $\N$. For each $x\in X$, there is a path ending at $x$. Consider the paths $\pi_1, \pi_2, \ldots, \pi_p$ not ending at an element in $X$. Note that $p=p(\N)$. Since the paths are vertex disjoint and partition $V$, the set of paths forms a spanning sub-forest of $\N$. So, by adding, for each path, one edge of $\N$ directed into the starting vertex, we construct a rooted spanning tree $T$ of $\N$. The leaves of $T$ not in $X$ are precisely the last vertices of the paths $\pi_1, \pi_2, \ldots, \pi_p$. Since there are $p(\N)$ of these paths, it follows that $l(\N)\leq p(\N)$.

\blue{We next show that $p(\N)\le t(\N)$. Let $\N'$ be a tree-based network that is obtained from $\N$ by attaching $t(\N)$ leaves. Let $T$ be a base tree for $\N'$, and let $U$ denote the leaf set of $T$. If we now apply Lemma~\ref{l:trees.M} with the same choice of $U$ and $T$, then $T$ can be partitioned into at most $|U|=t(N)+|X|$ paths each of which ends at an element in $U$. Thus $p(\N)\le t(\N)$.}

\blue{Lastly, we show that $t(\N)\le l(\N)$. Let $T$ be a rooted spanning tree of $\N$ that realises $l(\N)$. For each leaf $\ell$ of $T$ that is not in $X$, attach a new leaf to an edge directed out of $\ell$. If $\ell$ is a tree vertex of $\N$, then choose arbitrarily one of the outgoing edges to attach the new leaf. Let $\N'$ denote the resulting phylogenetic network. Since $T$ is a rooted spanning tree of $\N$, it is easily seen that we can extend $T$ to give a rooted spanning tree of $\N'$ whose leaf set coincides with the leaf set of $\N'$. Hence $\N'$ is tree-based, and it follows that $t(\N)\le l(\N)$. This completes the proof of the theorem.}
\end{proof}

\subsection{Computational complexity and explicit constructions}

We now consider the time to compute each of the three measures. Let $\N$ be a phylogenetic network on $X$ and let $n$ denote the total number of vertices in $\N$. Since each vertex in $\N$ has degree at most three, $\N$ has $O(n)$ edges. By Theorem~\ref{matchthm},
\blue{$$l(\N)=p(\N)=t(\N)$$}
and, by Lemma~\ref{unmatched},
$$p(\N)=u(\G_{\N})-|X|.$$
Thus $l(\N)$, $p(\N)$, \blue{and $t(\N)$} can all be computed by finding a maximum-sized matching in $\G_{\N}$. Since $\G_{\N}$ has $2n$ vertices and the same number of edges as $\N$, we can find such a matching in time $O\left(n^{3/2}\right)$~\cite{hop73}. Furthermore, as $\G_{\N}$ can be constructed in time polynomial in $n$, we can compute $l(\N)$, $p(\N)$, \blue{and $t(\N)$} in time polynomial in $n$.

%Lastly, finding a maximum-sized matching in a bipartite graph can be done in time $O(\sqrt{n}m)$, where $n$ is the number of vertices and $m$ is the number of edges of the bipartite graph \cite{hop73}.

%Now consider the computation of $t(\N)$. By Theorem~\ref{attach}, $t(\N)$ can be computed by finding a maximum-sized matching in $\Z_{\N}$. The number of vertices and edges in $\Z_{\N}$ is each at most $O(n)$, and so we can find a maximum-sized matching in $\Z_{\N}$ in time $O\left(n^{3/2}\right)$. Again, $\Z_{\N}$ can be constructed in time polynomial in $n$ and so $t(\N)$ can be computed in time polynomial in $n$.

The proofs of Lemma~\ref{unmatched} \blue{and} Theorem~\ref{matchthm} implicitly establish how one can construct a rooted spanning tree, a system of vertex disjoint paths, and a tree-based network realising $l(\N)$, $p(\N)$, and $t(\N)$. We end this section with explicit algorithms for each of these constructions. Their correctness is omitted as this is essentially done in the proofs of these results. The input to each algorithm is a phylogenetic network $\N$ on $X$. The first algorithm constructs a minimum-sized set $\cP$ of vertex disjoint paths that partition the vertices of $\N$.

\noindent{\sc Vertex Disjoint Paths} ($\N$)
\begin{enumerate}[1.]
\item Construct $\G_{\N}$ and find a maximum-sized matching of $\G_{\N}$.

\item Let $U_2$ denote the set of unmatched vertices in $V_2$.

\item For each $u_0\in U_2$, find the unique maximal sequence
$$(u_0, u_1), (u_1, u_2), (u_2, u_3), \ldots, (u_{k-1}, u_k)$$
of matched edges in $\G_{\N}$ and set $P_{u_0}$ to be the path $u_0\, u_1\, u_2\, \cdots\, u_k$.

\item Set $\cP=\{P_{u_0}: u_0\in U_2\}$ and return $\cP$.
\end{enumerate}

The second algorithm constructs a rooted spanning tree $T$ of $\N$ that minimises the number of leaves in $V-X$, where $V$ is the vertex set of $\N$.

\noindent{\sc Rooted Spanning Tree} ($\N$)
\begin{enumerate}[1.]
\item Let $\cP$ be the set of vertex disjoint paths returned by a call to {\sc Vertex Disjoint Paths}~($\N$).

\item Let $\pi_{\rho}$ denote the path in $\cP$ traversing the root $\rho$ of $\N$.

\item For each path $\pi=u_0\, u_1\, u_2\, \cdots\, u_k$ in $\cP-\{\pi_{\rho}\}$, extend $\pi$ to
$$\pi'=w\, u_0\, u_1\, u_2\, \cdots\, u_k,$$
where $(w, u_0)$ is an edge in $\N$.

\item Set $E_{\pi_{\rho}}$ to be the edge set of $\pi_{\rho}$ and, for all $\pi\in \cP-\{\pi_{\rho}\}$, set $E_{\pi'}$ to the edge set of $\pi'$.

\item Set $T=(V, E)$, where
$$E=E_{\pi_{\rho}}\bigcup_{\pi\in \cP-\{\pi_{\rho}\}}E_{\pi'}$$
and return $T$.
\end{enumerate}

The last algorithm constructs a tree-based network $\N'$ from $\N$ by attaching $t(\N)$ leaves.

\noindent{\sc Tree-Based Network} ($\N$)
\begin{enumerate}[1.]
\item \blue{Let $T$ be the rooted spanning tree returned by a call to {\sc Rooted Spanning Tree~($\N$)}.}

\item \blue{Let $L$ denote the subset of leaves in $T$ not in $X$.}

\item \blue{For each leaf $\ell\in L$, attach a new leaf to an edge directed out of $\ell$ in $\N$.}

\item \blue{Set $\N'$ to be the resulting network and return $\N'$.}
\end{enumerate}

%\begin{enumerate}[1.]
%\item Construct $\Z_{\N}$ and find a maximum-sized matching of $\Z_{\N}$.
%
%\item Let $U_R$ denote the set of unmatched reticulations in $\Z_{\N}$.
%
%\item For each $r\in U_R$, attach a new leaf to a reticulation edge directed into $r$ in $\N$.
%
%\item Set $\N'$ to be the resulting network and return $\N'$.
%\end{enumerate}

\section{Discussion and Further Questions}
\label{ending}

In this paper, we have established several new characterisations of tree-based networks based on notions of antichains, path partitions and bipartite matchings. We then applied these results in Section~\ref{deviation} to define and analyse \blue{three equivalent} and computable measures that quantify how close to being tree-based an arbitrary network is.

There are other ways to quantify the extent to which an arbitrary phylogenetic network deviates from being tree-based;  two particular alternative  measures are the following:
\begin{enumerate}[1.]
\item The minimum number of vertices in $\N$ that are required to be absent from any rooted tree that is embedded in $\N$ and has the same leaf set as $\N$.

\item The minimum number of rooted trees with the same leaf set as $\N$ that are required to be embedded in $\N$ in order that every vertex of $\N$ is present in at least one of the trees. 
\end{enumerate}
The first measure is non-negative and equal to zero if and only if $\N$ is tree-based, while the second measure is at least $1$, and equal to $1$
if and only if $\N$ is tree-based.
Determining the computational complexity of computing these two measures seems an interesting topic for future work. 

A further project would be to generalise our results concerning the three deviation indices to the larger class of  {\em non-binary} phylogenetic networks, in which vertices are allowed to have  in-degree and out-degree greater than two. There are at least two natural ways to extend the definition of tree-based to this more general class of networks, as described and studied in \cite{jet16}.

%[COMMENT from Mike: The paper \cite{bel07} (copy in BOX as partition\_chapter.pdf) may be relevant  to this section -- see the two highlighted text sections.]

%\section*{References}

%
%\bibliographystyle{elsarticle-num}
%%%%%%%%%%%%%%%%%%%%%%%%
%\section*{References}
%
%\bibliography{tbn}

\end{document}